\documentclass[a4paper,12pt]{amsart}
\pdfoutput=1
\usepackage[protrusion=true,expansion=true]{microtype}
\usepackage[colorlinks=true,linkcolor=black,citecolor=black]{hyperref}
\usepackage[T1]{fontenc} 
\usepackage{mathtools}
\usepackage[all]{xy}
\usepackage[english]{babel}
\usepackage{frcursive}          
\usepackage{tikz}
\usetikzlibrary{patterns}
\usepackage{eucal}
\usepackage{graphicx}
\usepackage{epsfig}
\usepackage{mathrsfs}
\usepackage{amssymb}
\usepackage{amsxtra}
\usepackage{enumerate}

\newtheorem{theorem}{Theorem}[section]

\newtheorem{proposition}[theorem]{Proposition}
\newtheorem{lemma}[theorem]{Lemma}

\newtheorem{definition}[theorem]{Definition}

\theoremstyle{remark}

\def \1{\mathbb {1}}
\def \RM{\mathbb {R}}
\def \NM{\mathbb{N}}
\def \ZM{\mathbb{Z}}

\def \CM{\mathbb{C}}

\def \Der {{\rm Der\,}}

\def \p {{\rm exp\,}}

\def \d{\partial}
\def\dt{\delta} 
\def\a{\alpha}
\def\b{\beta}
\def\e{\varepsilon}  
\def\g{\gamma}

\def\l{\lambda}

\def\p{\varphi}

\def \s{\sigma}

\def \to{\longrightarrow} 
\def \w{\wedge}

\def \< {{\langle }}
\def \> {{\rangle }}
\def \( {\left( }
\def \) {\right) }

\newcommand{\Bt}{{\mathcal B}}

\newcommand{\Ot}{{\mathcal O}}

\title[Fixed point theorems]{Fixed point theorems \\ for small denominators problems}
\author{Mauricio GARAY \MakeLowercase{ and} Duco van Straten.}
\begin{document}


\begin{abstract}
In the seventies, Zehnder established a Nash-Moser type implicit function theorem in the analytic setting. This theorem has found many applications in dynamical systems, although its use typically requires substantial technical effort. We further develop  the analytic theory and give fixed point theorems with direct applications to dynamical systems. \end{abstract}
\maketitle
 
\section*{Introduction}
The aim of this paper is to establish fixed point theorems for Banach scales of holomorphic functions, thereby providing tools tailored for direct applications to small divisor problems, in particular to KAM theory.

 In applications of functional analysis, there is a strong emphasis on Sobolev or H\"older spaces
while Banach spaces of holomorphic functions usually do not play a central role. However, Zehnder 
had already given  Nash-Moser implicit function theorems in the 
analytic setting~\cite{Zehnder1974implicit,Zehnder_implicit}. Although this leaves out the more general case of smooth functions, for applications in celestial mechanics and dynamical systems this restriction to analytic functions does not really limit practical applicability. 

The basic idea, used by Zehnder, to implement the Newton iteration in a Banach scale can be traced back to Nagumo's 
proof of the Cauchy-Kovalevskaya theorem~\cite{Nagumo}. One considers the Banach scale  $\Ot^b(D_s)$  of bounded functions holomorphic in a disc of radius $s$ with the supremum norm. This technique was rediscovered by Kolmogorov and is now widely used in dynamical systems~\cite{Kolmogorov_KAM}. 

In the analytic category one does not need 
to consider smoothing operators in order to implement Newton-type iterations. For instance, the maps 
$$\Ot^b(D_{s_n}) \to \Ot^b(D_{s_{n+1}}),\ f \mapsto f'$$ 
may be iterated indefinitely by choosing any positive decreasing sequence $(s_n)$ converging to a positive limit.   

 Zehnder already noted in his first paper that the situation is much simpler in the analytic case: { \em  The proof of this theorem is elementary; the real difficulty, however, lies in showing that the
assumptions can be met in the relevant applications \cite{Zehnder1974implicit}.}

 Indeed, applying Zehnder’s implicit function theorems remains a real challenge.  Zehnder  was able to prove a normal form theorem for almost
 linear vector fields on the torus, but it took 12 years before Bost and Herman gave a proof of the weak KAM theorem--the Kolmogorov invariant
 torus theorem--using implicit function theorems~\cite{Bost_KAM}.   The Bost-Herman proof was lengthy and delicate but in~\cite{Fejoz_KAM}, Féjoz gave a proof, in the analytic category,  based on an elementary fixed point theorem in the spirit of Zehnder's original work.  More recently, it has been shown by Alazard and Shao that the Kolmogorov invariant torus theorem can be derived from the standard Banach fixed point theorem~\cite{Alazard_Shao}.

 As far as we know, there is no complete proof of the strong KAM theorem in the $C^\infty$ case, that is, of the existence of a positive measure set of invariant tori in almost integrable systems as a direct application of implicit function theorems. In fact, one can hardly qualify the Bost-Herman proof as a direct application either.

 Note that in the literature {\em KAM theorem} may refer to three different results which in increasing order of difficulty are: {\em normal forms of quasi-linear vector fields on the torus}---the baby model--- {\em the invariant torus theorem}---the weak KAM theorem--- {\em the existence of a positive measure set of invariant tori}---the strong KAM theorem---, which we simply call the KAM theorem.
   
  The theorems that we will present are all based on the idea of {\em scalar models} where every infinite dimensional iteration is reduced to studying a simple real sequence. Namely, if $(x_n)$ is a sequence such that  $x_n \in E_n$, where $(E_n,|\cdot|_n)$ is an increasing sequence of Banach subspaces of a locally convex space $E$
  then we replace the relation $$|x_{n+1}-x_{n}|_{n+1} \leq |f(x_{n})-f(x_{n-1})|_n,\ x_{n+1}=f(x_n) $$ by a scalar model
  $$u_{n+1}=g(u_n) \text{ where } |f(x_{n})-f(x_{n-1})|_n \leq g( u_n),\ u_n=|x_n-x_{n-1}|_n  .$$
  If the sequence $(u_n)$ happens to be summable then the iteration will converge. This is reminiscent of Cauchy's majorant method.
   
  The structure of the paper is as follows.\\
In §1, we start by recalling Hamilton's inverse function theorem and explain some of its limitations~\cite{Hamilton_implicit} (see also~\cite{Sergeraert}). We then describe what type of small denominators one should expect. These are more general than the ones considered in the papers of Zehnder~\cite{Zehnder1974implicit,Zehnder_implicit}. We then show the convergence of the Newton method in an abstract context and deduce an analytic variant of Hamilton's inverse function theorem~(Theorem~\ref{T::Hamilton}).
 Finally, we recover a slight generalisation of Zehnder's implicit function theorem~\cite{Zehnder1974implicit}. The proofs are based on the convergence of the Newton-type scalar model $u_{n+1}=a_n u_n^2$ where $(a_n)$ is a Bruno sequence~(\ref{P::Bruno}). The same model can be used for the weak KAM theorem.

In §2, we aim at establishing theorems with direct applications to the local theory of dynamical systems such as the Bruno--Siegel linearisation theorem for vector fields, the general versal deformation theorem for vector fields or the proof of the Herman invariant tori conjecture~\cite{Bruno,Versal_fields,Herman_conjecture,Siegel_vecteurs}. This type of theorem involves an increasing filtration by the order of perturbation, as in the ordinary Morse lemma, during the iteration. So we start by examining the Morse lemma in our context and then establish a variant of the Banach fixed point theorem adapted to perturbation theory. Surprisingly enough, in the perturbative context, the quadraticity of Newton-type iteration is not needed. The scalar model  for the perturbative case is simply $u_{n+1}=b_n u_n$, where $b$ is a decreasing Bruno sequence and therefore the convergence
is straightforward. Although elementary, we were ignorant of this fact for a long time.

In §3, we consider KAM theory in the spirit of Arnold: contrary to classical perturbation theory, there is in this case no filtration preserved by the iteration~\cite{Arnold_KAM}. To explain this fact---which is often neglected---we consider a baby example. We then prove a fixed point theorem which we used in our proof of the KAM theorem for symplectic tori~\cite{Symplectic_torus}. It can be used to simplify the proof of the strong KAM theorem as well. In this case, small denominators are compensated by the quadraticity of the Newton iteration but the approximation by truncation of the harmonics involves
a new factor. Our finite dimensional model mixes both the Newton iteration and the perturbative case: it is of the form $u_{n+1}=a_n u_n^2+b_nu_n$ where $(a_n)$ and $(b_n)$ are respectively increasing and decreasing sequences. In practice, assuming Bruno conditions on the frequencies,
our condition is satisfied. Our theorem shows that the KAM theorem
holds under Bruno conditions, whereas in~\cite{Symplectic_torus}
we assumed a stronger condition. 

\section{An analytic inverse function theorem}
We start by discussing and proving an analytic variant of the Hamilton inverse function theorem. We will see that the analytic case turns out to be much simpler than the $C^\infty$ one.
\subsection{The Hamilton inverse function theorem}

\begin{theorem}[\cite{Hamilton_implicit}]
Let $E$ and $F$ be tame Fréchet spaces, let $U \subset E$ be an open subset, and let $f \colon U \to F$ be a smooth tame map.
Suppose that the equation for the derivative $Df(x) h = k$
has a unique solution $h = Vf(x) k$ for all $x \in U$ and all $k \in F$, and that the family of inverses $Vf \colon U \times F \to E$ is a smooth tame map. Then $f$ is locally invertible, and each local inverse $f^{-1}$ is a smooth tame map.

\end{theorem}
We will not recall the definition of tame Fr\'echet space,  as we consider only the case $E=F=C^\infty([0,1],\RM)$ with $C^k$ norms: 
$$\| \p \|_k=\max_{|j| \leq k}\sup_{x \in [0,1]}|\d^j \p(x)| .$$ That a map $f$ is tame means that there exists for any $n \geq 0$
 constants $C_n$ such that for any $\p \in E$:
 $$(*)\ \|f(\p)\|_n \leq C_n(1+\| \p \|_{n+r}). $$
 It is a way of saying that there is a loss of regularity of $r$-derivatives when we apply $f$. For linear operators tameness is defined by
  $$ \|L(\p)\|_n \leq C_n(\| \p \|_{n+r}). $$
  For instance, a linear operator of order $r$ satisfies this property.
 Tameness for the family of inverses means that
 $$ \| Vf(\p) k \|_n \leq C_n(1+\| \p \|_{n+r})\| k \|_{n+r}$$
 and smooth tameness means that all derivatives of the map are also tame.

 \subsection{The analytic case}

 Let $E=\CM\{ z \}$ be the space of convergent power series in one variable.
  Let $E_t$  be the Banach space of  $L^2$ holomorphic function inside the disk 
 $D_t= \{ z \in \CM: | z | < t \} $:
 $$| f |_t= \sqrt{\int_{D_t} |f(z)|^2 dx \w dy },\ z=x+iy.$$
 The family $(E_t)$ forms a Banach scale that is a decreasing family of Banach spaces parametrised by $\RM_{>0}$: $E_t \subset E_s$
 for $t>s$. The space $E$ is then the direct limit of the Banach spaces $E_t$ (see for instance~\cite{Grothendieck_EVT}).

 \subsection{What loss of regularity do we expect?}
 The above example is particularly simple but, in general, solving linearised equations will exhibit loss of regularity in the analytic case as well.
 There are mainly two reasons for loosing regularity: differentiation and division. Say if $\p$ is a continuous function vanishing at the origin then $\p(x)/x$ might be non-continuous~(for instance if
  $\p(x)=|x|$). 

In the analytic case,
 the Cauchy integral formula implies the {\em Cauchy} inequalities:
$$\sup_{|z|=s}|f^{(k)}(z)| \leq \frac{k!}{(t-s)^k}\sup_{|z|=t }|f(z)| $$
which shows the loss of regularity given by differentiation. 
While for the division process we have, by the maximum principle (assuming the quotient to be holomorphic):
$$\sup_{|z| \leq t}\left|\frac{f(z)}{z^k}\right| \leq \frac{1}{t^k}\sup_{|z|=t }|f(z)| .$$
 As a general rule, loss of regularity, in the analytic context, involves factors of the form $s^\a t^\b(t-s)^\g$, $\a,\b,\g \geq 0$  in the denominator.

%
 \subsection{Local maps}
 We may now give the analog of Hamilton's tameness condition in the analytic context.
An decreasing Banach chain $(E_t), t \in \RM_{>0}$ of a locally convex space $E$ is called {\em Kolmogorov} if the inclusions $E_t \subset E_s$ have norm $<1$ and $E=\bigcup_t E_t$. We abuse notations and use indifferently the notation $E$ for the chain and for the locally convex space.
We also extend the norm of $E_t$ to $E$ by putting it equal to $\infty$ for elements which are not in $E_t$.
\begin{definition} A map $f:E \to F$ between Kolmogorov chains is called {\em local} if $f(E_t) \subset F_s$ for any $s<t$ and if
 there exists a function, called a {\em local factor}, of the form $M(s,t)=Cs^{-\a} (t-s)^{-\b}$, $\a,\b \geq 0$, such that
 $|f(x)|_s \leq M(s,t)(1+|x|_t) $ for any $x \in E_t$. 
 \end{definition}
The notion applies also for linear map for which locality is defined by $|L(\xi)|_s \leq M(s,t)|\xi|_t $ for any $x \in E_t$. If a map $(x,\xi) \mapsto f(x,\xi)$ is non linear in $x$ and linear in $\xi$ then locality means that
$$|f(x,\xi)|_s  \leq  M(s,t)(1+|x|_t)|\xi|_t$$

\subsection{The classical Newton iteration}
Before going further, let us recall  the Newton iteration
for one variable functions $f:\RM \to \RM$. It is defined by
$$x_{n+1}=x_n-\frac{f(x_n)}{f'(x_n)} $$
and if we assume that
\begin{enumerate}[{\rm 1)}]
\item $x_0$ is close enough to a solution,
\item $f'>A$, $f''<2m$
\end{enumerate} then the iteration converges to a solution, so from an approximate solution $x_0$, we deduce a real solution.
 Let us recall the proof:
\begin{align*}
|x_{n+1}-x_{n}|&=\left|\frac{f(x_{n})}{f'(x_n)}\right| \leq A^{-1} |f(x_{n})|=A^{-1} |f(x_{n-1}+(x_n-x_{n-1}))|\\
& =A^{-1} |f(x_{n-1})+f'(x_{n-1})(x_n-x_{n-1})+f''(\xi)\frac{(x_n-x_{n-1})^2}{2}| \\
& =A^{-1} \left|f''(\xi)\frac{(x_n-x_{n-1})^2}{2}\right|\\
& \leq mA^{-1}|x_n-x_{n-1}|^2=\left(K|x_n-x_{n-1}|\right)^2
\end{align*}
with $K=m^{1/2}A^{-1/2}$.
So it converges if $|x_1-x_0|=|\frac{f(x_0)}{f'(x_0)}| < m^{-1/2}A^{1/2}$ and therefore if $x_0$ is an approximate solution in the sense that
$|f(x_0)| < m^{-1/2}A^{3/2}$ then the iterates converge to a solution.

If we apply the method to $x \mapsto f(x)=g(x)-y$, $g(0)=0$, then we get a (one-variable) inverse function theorem, namely that $g$ is invertible
over the interval $g^{-1}(I)$ with $I=[-m^{-1/2}A^{3/2},m^{-1/2}A^{3/2}]$. If we apply the methode to a two variables function 
$$(\l,x) \mapsto f(\l,x),\ f(0)=0$$ and regard $\l$ as a parameter then we get an implicit function theorem over the interval $f^{-1}(I)$.

\subsection{Bruno sequences}
\label{SS::Bruno}
The convergence of the Newton iteration in the analytic context relies entirely on a finite dimensional model that we now explain. 
  If we define $ u_0=|x_{1}-x_{0}|$, then the Newton iterates are majorated by the sequence
\[ u_{n+1} = a u_n^2\]
where $a=mA^{-1}$. Due to the loss of regularity, we consider more general iterations of the form:
\[(*)\qquad u_{n+1} = a_n u_n^2\]
for some sequence $a=(a_n)$, $0<a<1$. We will show that this model captures the complexity both of Hamilton's inverse function theorem and
Zehnder's implicit function theorem in the analytic context.

The iteration $(*)$  leads to the solution
\[ u_{n+1} =(a_0^{1/2} a_1^{1/2^2}\ldots a_n^{1/2^{n+1}} u_0)^{2^{n+1}}.\]
If we define the product 
\[a_{\pi}:=\prod_{k=0}^{\infty} a_k^{1/2^{k+1}} \in ]0,+\infty[,\]
then:
\begin{proposition}\label{P::Bruno} If $(a_n)$ is a positive real sequence satisfying
\[a_{\pi}:=\prod_{k=0}^{\infty} a_k^{1/2^{k+1}} < \infty,\]
then the sequence defined inductively by $u_{n+1} = a_n u_n^2$
converges to zero if and only if $u_0 < a_\pi^{-1}$.
\end{proposition}

\begin{definition}[\cite{Bruno}]
1) Given a sequence $(a_n)$, we denote by $\hat a$ the sequence with terms $\widehat a_n:=a_0^{1/2} a_1^{1/2^2}\ldots a_n^{1/2^{n+1}}$
and call it the Bruno transform of $a$.\\
2) A strictly positive monotone sequence $a$ is called a  {\em Bruno sequence} if its Bruno transform
converges to a strictly positive number or equivalently if
$$\sum_{k\geq 0} \left| \frac{\log a_k}{2^{k}}\right| <+\infty. $$
\end{definition} 
The quantity $u_n= 2^{-n}\log a_n $ we call the {\em phase} of the Bruno sequence. A sequence is Bruno if and only if its phase is summable:
$$a_n=e^{2^n u_n} \in \Bt \iff  \sum_{n \geq 0} |u_n|<+\infty $$
 We denote respectively by $\Bt^+$ and $\Bt^{-}$ the set 
 Bruno sequences with positive and negative phase.
 
 The following lemma plays an important role:
 \begin{lemma}\label{L::absorb}  If  a Bruno sequence $\rho \in \Bt^-$ satisfies $\rho<1/2$  we have
$$1-\rho_n^{1/2^n} \ge \frac{1}{2^{n+1}}$$
\end{lemma}
\begin{proof}
The inequality $(1-\frac{x}{m})^m \ge (1-x), \;(m \ge 1)$ implies that
$(1-\frac{x}{2^n}) \geq (1-x)^{1/2^n},$
which for $x=1/2$ transposes to
\[ 1-2^{-1/2^n} \ge \frac{1}{2^{n+1}}.\]
As $\rho<1/2$, we have:
\[1-\rho_n^{1/2^n} \ge 1-2^{-1/2^n} \ge \frac{1}{2^{n+1}} \]
\end{proof}
\begin{lemma}\label{L::estimate} Let $M$ be a local factor. For any $t>0$, for any Bruno sequence $\rho \in \Bt^-$ such that $\rho<1/2$,
the sequence with terms  $M(s_{n+1},s_n)$, $s=t\hat \rho$,
 is bounded by a geometric sequence.

\end{lemma}
 \begin{proof}
Let $M(s,t)=s^{-\a}(t-s)^{-\b}$, we have   $s_{n+1}=\rho_n^{1/2^{n+1}}s_n$, $s_0=t$.
 Using Lemma~\ref{L::absorb}, we get that:
 $$M(s_{n+1},s_n)=s_{n+1}^{-\a-\b}(1-\rho_n^{1/2^{n+1}})^{-\b} \leq 2^{\b(n+2)}\rho_\pi^{-\a-\b}t^{-\a-\b} $$
 \end{proof}
  

 \subsection{Analytic version of the Hamilton theorem}
 
 Let $U \subset E$ be an open subset.
 A local map $f:U \to F$ is called {\em $C^k$-local} if the  induced maps
 $f(s,t):U_t \to F_s $ are $C^k$ and if the derivatives  as maps $U \times E^j \to F$ are local.
  \begin{theorem}\label{T::Hamilton} Let $(E_t)$ and $(F_t)$ be Kolmogorov chains, $  U=\{ x \in E:\exists t,\ |x|_t<1 \}$ and $ f:U\to F$ a $C^2$-local map. 
Suppose that for each $x  \in  U$ the linearization $ Df(x):E\to F$ is invertible, and the family of inverses, as a map 
$ L:U\times F\to E$ is local then $f$ is locally invertible. 
  \end{theorem}
 \begin{proof}
Assume for simplicity that $f(0)=0$, define $x_0=0$
and consider the Newton iterates
 $$x_{n+1}=x_n-L(x_n)(f(x_n)-y),\ y \in F  .$$
We show that they converge if $|y|_t$ is small enough for some $t$.
 For any decreasing Bruno sequence $(s_n), s_0=t$ (for instance a geometric sequence), we define $| \cdot |_n$ to be the supremum norm in $E_{s_n}$.
 Assume by induction that $\sum_{k=0}^{n-1} |x_{k+1}-x_k|_{2k+2} \leq 1-2^{-n}$, then we have the estimates
 \begin{align*}
|x_{n+1}-x_{n}|_{2n+2} &=|L(x_n)(f(x_{n})-y)|_{2n+2} \\
 &\leq M (s_{2n+2},s_{2n+1})(1+|x_n|_{2n+1}) |f(x_{n})-y|_{2n+1} \\
  & \leq 2 M (s_{2n+2},s_{2n+1}) |f(x_{n})-y|_{2n+1} 
  \end{align*}
As in the one dimensional case, we apply the Taylor formula at second order with increment $h_n=x_n-x_{n-1}$ and use locality of the second derivative:
\begin{align*}
 |f(x_{n})-y|_{2n+1} &= |f(x_{n-1}+h_n)-y|_{2n+1} \\
   &=|f(x_{n-1})-y+Df(x_{n-1})(h_n)+R(x_{n-1},x_n)(h_n,h_n) |_{2n+1}\\
 &= \left|R(x_{n-1},x_n)(h_n,h_n)\right|_{2n+1} \\ 
 & \le
\frac12
\sup_{\theta\in[0,1]}
\left\|
D^2f(x_{n-1}+\theta h_n)
\right\|_{E_{s_{2n}}^2\to F_{s_{2n+1}}}
|h_n|_{2n}^2.\\
 & \leq N(s_{2n+1},s_{2n})(1+|x_{n-1}|_{2n}+|x_{n}|_{2n})|x_n-x_{n-1}|_{2n}^2 \\
\end{align*}
where $M,N$ are local factors.  By Lemma~\ref{L::estimate}, we get an estimate of the form
$$|x_{n+1}-x_{n}|_{2n+2} \leq Cq^n|x_n-x_{n-1}|_{2n}^2$$
 $C>1,\ q>1 $.
As $x_1-x_0=L(x_0)(y)$ , iterating the above estimate we find that:
$$|x_{n+1}-x_{n}|_{2n+2} \leq  (A q^2 |y|_t)^{2^n} $$
for some constant $A$.
This shows that for $|y|_t$ small enough, we have uniformly $|x_{n+1}-x_n|_{2n+2} \leq 2^{-n-1}$ hence the map can be iterated and the sequence  $(x_n)$ converges. This concludes the proof of the theorem.
  \end{proof}
 In the same spirit, we can recover Zehnder's implicit function theorem as we shall now see.
  \subsection{The Zehnder implicit function theorem}
  Let $f:U \times V \to G$ be a map of Kolmogorov chains and $U=\{ \l \in E:\exists t,\ |\l|_t<1 \} \subset E$, $V= \{ x \in F:\exists t,\ |x|_t<1 \} \subset F$.
  Instead of an inverse for the derivatives Zehnder considers {\em quasi-inverses}, that is a local map:
 $$L:U \times V \times G \to E$$ 
  for which there exists local factor $M$ such that:
 $$|z-D_xf(\l,x) \circ L(\l,x)(z)|_s \leq M(s,t)|z|_t\, | f(\l,x)|_t $$
 \begin{theorem} \label{T::Zehnder} Assume that $f:U \times V \to G,\ 0 \mapsto 0$ is a $C^2$-local map of Kolmogorov chains and
the linearizations $ D_xf(\l,x):F\to G$ are quasi-invertible. Then there exists an open neighbourhood of the origin $U' \subset U$ and  a function
$g:U' \to F$ solving locally the equation $f=0$, that is, $f(\l,g(\l))=0$ for $\l \in U'$.
  \end{theorem}
 \begin{proof}
 The proof is almost identical to the inverse function theorem.
 In this case the Newton iterates take the form
 $$x_{n+1}=x_n-L(\l,x_n)(f(\l,x_n)), $$
 where $L$ is a quasi-inverse.
 We have similar estimates (we keep the same notations):
 \begin{align*}
|x_{n+1}-x_{n}|_{2n+2} &=|L(\l,x_n)(f(\l,x_{n}))|_{2n+2}  \leq M (s_{2n+2},s_{2n+1})  |f(\l,x_{n})|_{2n+1} \\
 |f(\l,x_{n})|_{2n+1} &=|f(\l,x_{n-1})+D_xf(\l,x_{n-1})(h_n)+R(\l,x_{n-1},x_n)(h_n,h_n) |_{2n+1}\\
 &\leq \left|f(\l,x_{n-1})-D_xf(\l,x_{n-1})L(\l,x_{n-1})(f(\l,x_{n-1}))\right|_{2n+1}\\
 & \qquad +|R(\l,x_{n-1},x_n)(h_n,h_n)|_{2n}\\ 
 & \leq N(s_{2n+1},s_{2n})|x_n-x_{n-1}|_{2n}^2
\end{align*}
where $M,N$ are local factors. Hence the iteration converges by Lemma~\ref{L::estimate}.
  \end{proof}
    \section{A contraction mapping theorem} 
    
   Our first aim is to construct a baby model for perturbation theory, that is a model for step by step normal forms obtained by increasing order of approximation. This is the case, for instance, when we linearise a vector field: we first suppress quadratic terms then cubic terms etc~(see e.g.~\cite[Chapter 5]{Arnold_edo}).
   The one-dimensional Morse lemma gives a simple example in which this procedure can be applied and this will give our baby model. Then we give a fixed point theorem which ensures the convergence of perturbative normal forms. This result is used in~\cite{Herman_conjecture}.
 \subsection{The Morse lemma } 
 \label{SS::Morse}
 We consider the $\CM$-algebra $A=\CM[[ x ]]$ of formal power series in one variable. We want to eliminate the cubic term of the function 
  $$f(x)=\frac{1}{2}x^2+x^3$$
   by a local change of variable, that is by an automorphism of $A$. This can be done in an obvious way:
  \begin{align*}
  f(x)&=\frac{1}{2}x^2(1+2x)\\
       &=\frac{1}{2}y^2 \text{ with } y=x \sqrt{1+2x}
  \end{align*}
  Instead, we want to illustrate the idea of perturbative expansions with this example. By taking the exponential of the derivation
  $$v_0=-x^2\d_x \in \Der(A) $$
  we eliminate the cubic term:
   \begin{align*}
   e^{v_0} f&=f+v_0(f)+\frac{1}{2!} v_0^2(f)+\dots \\
       &=\frac{1}{2}x^2 -\frac{3}{2}x^4+5x^5-\frac{115}{24}x^6+\frac{119}{96}x^7+\dots
  \end{align*}
  Then defining $f_0=f$, $f_1=e^{v_0}f_0$, we may repeat the process by eliminating the terms in $x^4$ and $x^5$. So at the next step, we define:
   $$v_1=\left(\frac{3}{2}x^3-5x^4\right)\d_x $$
   and get
   $$f_2=e^{v_1}f_1=\frac{1}{2}x^2-\frac{223}{24}x^6+\frac{3359}{96}x^7+\dots $$
   At the next step we may eliminate the terms $x^6,\ x^7,\ x^8,\ x^9$. More generally at step $n$, we have
   $$f_n=\frac{1}{2}x^2+R_n,\ R_n:=x^{2^n+2}\sum_{k \geq 0}a_kx^k $$
   We write $v_n=a_n\d_x$ with
   $$a_n=-x^{2^n+1}\sum_{k = 0}^{2^n-1}a_kx^k $$ There are two remarkable aspects
   \begin{itemize}
   \item The remainder term $R_n$ is of order $2^n+2$.
   \item The coefficient $a_n$ of the derivation $v_n=a_n\d_x$ is obtained after division by $x$ of the first $2^n$ terms in $R_n$.
   \end{itemize}
   The sets $D_t=\{ z \in \CM: |z | <t \}$  define a Kolmogorov chain
  $\Ot^b(D_t)$  for the supremum norm.
  Assuming the remainder $R_n$ is of the form $x^{2^n}S_n(x)$, then by the maximum principle
  $$| R_n |_s =s^{2^n}|S_n|_s \leq s^{2^n}|S_n|_t= \left(\frac{s}{t}\right)^{2^n+2}|R_n|_t,\ s<t $$
\subsection{Perturbative factors}
\label{SS::model}
We have seen that both implicit function and inverse function theorem are based on the simple model
$ u_{n+1}=Cq^n u_n^2$ and on the notion of local factors. We have seen in the example of the Morse lemma that the loss of regularity is
compensated by a factor of the form $(s/t)^{2^n}$, therefore we define:
\begin{definition} A {\em perturbative factor} is a sequence of function $(\l_n)$ 
of the form 
$$  \l_n(s,t)=a_n s^{-\a}(t-s)^{-\b}\left(\frac{s}{t}\right)^{2^n} ,\ \a,\b \geq 0 .$$
with $a=(a_n) \in \Bt^+$.
\end{definition}
 
 \begin{proposition}
\label{P::perturbative} Let $\l$ be a perturbative factor.
For any $b=(b_n) \in \Bt^-$, there exists $r=r(b)>0$ which has the following property.  For  any $t \leq r$, there is a Bruno sequence 
$\rho \in \Bt^-$ such that
$\l_n(s_{n+1},s_n)  =o( b_n) $
where $s=t\hat \rho$.
\end{proposition}
\begin{proof}
Let
$$  \l_n(s,t)=a_n s^{-\a}(t-s)^{-\b}\left(\frac{s}{t}\right)^{2^n} ,\ \a,\b \geq 0 .$$
and define
$$\rho_n:=2^{-2\b(n+1)-2n}a_n^{-2}b_n^2$$ 
and assume $\rho<1/2$ (which we can do without loss of generality). As  $s_{n+1}=\rho_n^{1/2^{n+1}}s_n$, by Lemma~\ref{L::absorb}, we have the estimate:
\begin{align*}
\l_n(s_{n+1},s_n)&=a_n \rho_n^{1/2} s_{n+1}^{-\a}(s_n-s_{n+1})^{-\b} \\
  &=2^{-\b(n+1)-n}b_ns_{n+1}^{-\a}(s_n-s_{n+1})^{-\b}\\
  &\le 2^{-\b(n+1)-n+\beta(n+1)}
s_\infty^{-\alpha-\beta}b_n =  s_\infty^{-\a -\b} 2^{-n}b_n=o(b_n)
\end{align*}
 \end{proof}

  \subsection{Contraction mapping theorem}
We consider a sequence of Kolmogorov chains $E_n \subset E_{m},\ m \geq n$ and $m,n \in \NM \cup \{ \infty \}$.
In particular, there is a limiting locally convex space $E_\infty$ which contains all $E_n$'s. Moreover, we assume that the inclusions
$E_{n,t} \subset E_{m,s}$, $s<t$, $m>n$ all have norm $ \leq 1$. Such a data we call an {\em Arnold chain} and denote by $E$. In our baby model the spaces
$E_n$ are all equal to $\CM\{ z \}$.

   \begin{theorem} Let $E$ be an Arnold chain, $\l$ a perturbative factor and $f_n:X_n \to X_{n+1}$ a sequence of maps where $X_n \subset E_n$
   are arbitrary subsets.  Assume that $f$ is a contraction in the sense that for every $n \in \NM$ we have:
   $$|f_n(x)-f_n(y)|_s \leq \l_n(s,t)|x-y|_t.$$ For 
   any  $b \in \Bt^-$ there exists a Bruno sequence $\rho \in \Bt^-$ such that the iterates
   $$x_{n+1}=f_n(x_n),\ s=t\hat \rho$$ converge in $E_\infty$ and, moreover:
    $|x_{n+1}-x_n|_{s_{n+1}}<b_n$ for large $n$.
  \end{theorem}
  \begin{proof}
  Without loss of generality, we may assume that $b$ is summable and less than $1$.
  Choose $r>0$ like in Proposition~\ref{P::perturbative}. We then have:
  \begin{align*}
  | x_{n+1}-x_n|_{s_{n+1}}&=|f_n(x_n)- f_n(x_{n-1})|_{s_{n+1}}\\
   &\leq \l_n(s_{n+1},s_n)|x_n-x_{n-1}|_{s_n}  \\
   & \leq b_n |x_n-x_{n-1}|_{s_n}	\\
   & \leq b_n b_{n-1}|x_1-x_0|_{s_0}=o(b_n)
   \end{align*}
  
   which shows that $(x_n)$ is a Cauchy sequence of $X_\infty$ as $b$ is summable.
  \end{proof}

\section{KAM iterations}
 The iteration involved in Arnold invariant tori theorem is not, like the Newton case, a quadratic iteration: it also contains a linear term~\cite{Arnold_KAM}.
 Contrary to the approximated inverse considered by Zehnder, Arnold's approximated inverse ruins quadraticity so we need a more elaborate analysis.
 
 It is also not a classical perturbative Ansatz as there is no increasing order of approximation: at each order one suppresses Fourier harmonics but these re-appear automatically, a phenomenon
 that we called after Sisyphus in~\cite{Symplectic_torus}. As in the preceding case, we concentrate on the study of finite
 dimensional models and therefore on Bruno sequences.
\subsection{Baby model in KAM theory}
One possible generalisation of functions on a compact manifold is given by closed one-forms. As by the Poincar\'e lemma, any such one form is locally the
differential of a function. Consider the simple case of the one form $\a=(1+\e \cos(\theta))d\theta $ on the circle $S^1=\RM/(2\pi \ZM)$ which
we consider as a perturbation of $d\theta$.

The action of derivations on one forms is given by the Cartan formula:
$$L_v \a=d i_v \a+i_vd\a=d i_v \a $$
So if we write explicitly $v=f(\theta)\d_\theta$, $\a=a(\theta)d\theta $:
$$L_va(\theta)d\theta=(a(\theta)f'(\theta)+a'(\theta)f(\theta))d\theta. $$

Like for the Morse case, we solve the {\em homological equation}:
$$ L_vd\theta+\e \cos(\theta) d\theta=0$$
which gives $v=-\e \sin(\theta)  \d \theta$. Next we compute:
$$\a_1=e^{L_v}\a=\a+L_v\a+\frac{1}{2!} L_v^2\a+\dots $$
and find for the first two terms in the exponential:
$$\a_1=(1 -\frac{\e^3}{4}\cos(\theta) - \frac{\e^2}{2}\cos(2\theta) + \frac{3\e^3}{4}\cos(3\theta)+\dots)d\theta $$
  
 where the dots stand for terms involving higher powers of $L_v$. How to rearrange the terms in the expansion into principal part and reminder?
 We wish to decompose $\a_1$ in three parts:
$$\a_1=d\theta+\b_1+\g_1 $$
The term $\b_1$ is our "principal part" and the term $\g_1$ is the "negligible part". What this means depends on the context. 
There are essentially two possibilities:
\begin{enumerate}[{\rm 1)}]
\item $\e$ is a perturbative parameter, we filtrate the space with respect to the degree in $\e$ that is we neglect here the terms of $\e$-degree $>3$. In this   
\item   $\e$ is a fixed value (say $\e=1$), then we filtrate the space with respect to the degree of the harmonics, we neglect here
 harmonics of degree $>2$.
\end{enumerate}
These lead to two different choices:
\begin{enumerate}[{\rm 1)}]
\item  $\displaystyle{\b_1=-\frac{\e^2}{2}\cos(2\theta)-\frac{\e^3}{6}\cos(\theta) + \frac{\e^3}{2}\cos(3\theta),\  \g_1= \frac{\e^4}{24}\cos(2\theta)-\frac{3\e^4}{8}\cos(4\theta)\dots}  $
\item  $\displaystyle{\b_1'= (-\frac{\e^3}{6}+\dots)\cos(\theta) - (\frac{\e^2}{2}+\dots) \cos(2\theta),\ \g_1'= \frac{\e^3}{2}\cos(3\theta)+\dots}  $
\end{enumerate}
The first case corresponds to the {\em perturbative case} discussed earlier. 
The latter case, we call the {\em KAM case}. At each step we double the number of harmonics taken into account but the harmonics we tried to suppress re-appear at each step: the iteration does not preserve the filtration by the degree of the harmonics.
Arnold's idea is that they re-appear but are much smaller in norm~\cite{Arnold_KAM}.
For instance, in our first example, we failed in suppressing the $\cos(\theta)$ term, for $\e=1$, we have: 
$$\a_0=(1+\cos \theta)d\theta,\ \a_1=(1-\frac{1}{4}\cos \theta+\dots)d\theta  $$ 
so $\cos \theta$ remains but with a smaller coefficient ($-1/4$).
 \subsection{The Arnold-Moser lemma}
 We now wish to understand how the order of harmonics is involved  while computing norms and, more precisely, why is the truncation of high order harmonics harmless. For simplicity,
 we consider the one dimensional case.
 
  Consider the strip  
  $$S_t:=\{ \theta \in \CM: | \text{Im}\,\theta| \leq t,\ \text{Re}\,\theta \in [0,2\pi] \} $$
  and the space of $L^2$-functions on this strip with norm $| \cdot |_t$. We have
  $$\langle e^{ik\theta} | e^{ik\theta} \rangle_t=\frac{1}{2\pi} \int_{S_t} e^{ik\theta} \overline{e^{ik\theta}}=\left\{ 
  \begin{matrix} 2t & \text{ if }  k=0\\
  \frac{\sinh(2kt)}{k}& \text{ otherwise. }
  \end{matrix} \right. $$
  \begin{lemma}
 Assume that the function $\g_n:=\sum_{|k| \geq 2^n} a_k e^{ik\theta} $ is holomorphic in the strip $S_t$, then we have for $s<t$:
 $$ |\g_n|_s \leq q^{2^n} |\g_n|_t\text{ with } q=e^{s-t}<1.$$
 \end{lemma}
 \begin{proof}
  $$| \g_n |_s^2 =\sum_{|k| \geq 2^n} |a_k|^2 \frac{\sinh(2ks)}{k}  = \sum_{|k| \geq 2^n} |a_k|^2 \frac{\sinh(2ks)}{\sinh(2kt)}\frac{\sinh(2kt)}{k}$$
  To estimate the right-hand side, observe that the function
  $$\RM_{>0} \to \RM,\ x \mapsto \frac{\sinh(2xs)}{\sinh(2xt)} $$
  is decreasing hence
  $$| \g_n |_s^2  \leq \frac{\sinh(2^{n+1}s)}{\sinh(2^{n+1}t)}\sum_{|k| \geq 2^n} |a_k|^2\frac{\sinh(2kt)}{k}=\frac{\sinh(2^{n+1}s)}{\sinh(2^{n+1}t)}|\g_n|_t^2$$
  and moreover:
  $$\frac{\sinh(2^{n+1}s)}{\sinh(2^{n+1}t)}=\frac{e^{2^{n+1}s}(1-e^{-2^{n+2}s})}{e^{2^{n+1}t}(1-e^{-2^{n+2}t})} \leq \frac{e^{2^{n+1} s}}{e^{2^{n+1} t}} $$
  This proves the lemma.
 \end{proof}

\subsection{Tamed pairs}
Let us start by investigating mixed linear-quadratic
iteration of the form
$$u_{n+1}=\frac{1}{2}\left( a_n u_n^2+b_n u_n\right) $$

Of course, this general type of iteration is of great complexity, as 
it contains the famous logistic iteration $u_{n+1}=ru_n(1-u_n)$, with its
infinite richness, as subcase. Bruno conditions are not sufficient to ensure the convergence of this iteration.

\begin{definition}\label{D::pair}
Let $a=(a_n) \ge 1$ and $b =(b_n) $ be positive sequences. The pair $(a,b)$ is called a {\em tame pair}
if  
$$(\star)\ \exists N,\ \quad \forall n \geq N \in \NM,\;\;\; a_n b_{n}^2 \leq b_{n+1}. $$
\end{definition}
Clearly, if $(a,b)$ is a tame pair then so is $(a',b)$ with $a' \le a$.

The geometric  pair $(a,b)$ with $a_n=\l^n, b_n=\mu^n$ is tame, if $\l \mu \le 1$.
A typical tame pair $(a,b)$ is given by 
\[ a_n=e^{\a^n},\;\;\;b_n=  e^{-\b^n},\] 
where $\a,\b$ are both positive real numbers, $1 \le \a < \b \leq 2$. Note the boundary case $a_n=1, b_n=e^{-2^n}$.

The following proposition underlines the relevance of tame pairs for mixed
linear-quadratic iterations.

\begin{proposition}
\label{P::model}
Let $(a,b)$ be a tame pair. Then there exists $\dt>0$ such that for $u_0 \leq \dt$   the real sequence defined by the
recursion
$$u_{n+1}=\frac{1}{2} \left(a_nu_n^2+ b_n u_n \right)$$
converges to zero and one has the estimate
$u_n \leq   b_n,\ \forall n \in \NM.$
\end{proposition}
\begin{proof}
The maps $u_0 \mapsto u_{k},\ k \leq N$ are polynomial maps which vanish at $0$, hence
  for $\dt $ small enough
one has: 
$$u_0 \leq \dt \leq b_0 \implies u_1\leq  b_1,\  u_2\leq  b_2 \dots, u_N \leq  b_N$$
 Let us now show by induction that $u_n \leq b_n$.
As  $a \geq 1$, using $(\star)_N$, we get that: 
\begin{align*}
u_{n+1}&= \frac{1}{2} \left(a_nu_n^2+b_n u_n \right)
 \leq  \frac{1}{2} \left(a_nb_n^2+b_n^2\right) 
 \leq  a_nb_n^2   \leq   b_{n+1} 
\end{align*}
\end{proof}

\subsection{Tameness and small denominators}

 \begin{definition}A {\em KAM-factor} $(M,N)$ is a pair of sequences of functions $(M_n,N_n)$ such that $M_n(s,t)= a_n(t-s)^{-k}s^{-q}$ is a local factor and  $N_n(s,t)= b_n (t-s)^{-\ell}s^{-m}(\frac{e^s}{e^t})^{2^n}$
 where $a \in \Bt^{+}$, $b \in \Bt^-$ are Bruno sequences. 
 \end{definition}
 The {\em phases of the KAM factor} are defined as the phases of the sequences $a,b$.
  
 Recall from section \ref{SS::Bruno} that the Bruno transform of a sequence $\hat \rho $
 is defined by \[
\hat \rho_n \;:=\; \prod_{k=0}^n \rho_k^{\,2^{-(k+1)}} .
\]
and that the quantity $r_n=  2^{-n}\log \rho_n $ is called the {\em phase} of the Bruno sequence.

\begin{proposition}
\label{P::KAM}
Let $(M,N)$ be a KAM factor with phases $\a,\b$ such that 
$  \a_n+2(\b_n-\b_{n+1})=o(n^{-2-\e})$ for some $\e>0$ and consider the Bruno sequence $\rho \in \Bt^-$
with phase $1/n^{1+\e}$. For any $t>0$, 
  the pair $(M_n( s_{n+1},s_{n}), N_n(s_{n+1},s_{n}))$, $s_n=t\hat \rho_{n-1}$, is tame and for any Bruno sequence $c$
  with phase $\g_n=o(1/n^{1+\e})$ 
 we have $ N_n(s_{n+1},s_{n})< c_n$ for sufficiently large $n$. 
 \end{proposition}    
\begin{proof}
 Write $M_n(s,t)= a_n(t-s)^{-k}s^{-q},\ N_n(s,t)= b_n (t-s)^{-\ell}s^{-m}(\frac{e^s}{e^t})^{2^n}$ and define $r_n=n^{-1-\e}$.
As \[
s_{n+1} \;=\; \rho_n^{\,1/2^{n+1}}s_n=e^{-r_n/2}\, s_n,
\]we get the asymptotics for $n \to +\infty$:
\begin{align*}
 \frac{e^{s_{n+1}}}{e^{s_n}}&=e^{e^{-r_n/2}s_n-s_n} \\
 \log \frac{e^{s_{n+1}}}{e^{s_n}} &\sim  -\frac{r_n s_\infty}{2} \\
\left( \frac{e^{s_{n+1}}}{e^{s_n}}\right)^{2^n} &= e^{-2^{n-1}r_ns_n+o(2^nr_n)}
\end{align*}

Defining the sequence $\sigma$ by $\sigma_n:=s_n-s_{n+1}>0$, we have
\[
M_n(s_{n+1},s_n)=a_n \sigma_n^{-k}s_{n+1}^{-q},\qquad
N_n(s_{n+1},s_n)=b_n \sigma_n^{-\ell}s_{n+1}^{-m}
\left(\frac{e^{s_{n+1}}}{e^{s_n}}\right)^{2^n}.
\]

The tameness condition for the pair $(A_n,B_n):=(M_n(s_{n+1},s_n),N_n(s_{n+1},s_n))$
reads
\[
A_n\,B_n^2 \le B_{n+1}
\quad \text{for $n$ large enough.}
\]

Substituting the expressions of $A_n$ and $B_n$, we obtain
\begin{align*}
& a_n b_{n}^2\,
\sigma_n^{-k-2\ell}
\,s_{n+1}^{-q-2m}
\left(\frac{e^{s_{n+1}}}{e^{s_n}}\right)^{2^{n+1}}
\;\le\;
b_{n+1} \sigma_{n+1}^{-\ell}s_{n+2}^{-m}
\left(\frac{e^{s_{n+2}}}{e^{s_{n+1}}}\right)^{2^{n+1}}.
\end{align*}
Passing to the logarithm, we obtain
\begin{align*}
&2^n(\alpha_n+2\beta_n-2\beta_{n+1})
-(k+2\ell)\log \sigma_n
+\ell\log \sigma_{n+1} \\
&\qquad
-(q+2m)\log s_{n+1}
+m\log s_{n+2} \\
&\qquad\qquad
\le
2^{n+1}\bigl(s_n+s_{n+2}-2s_{n+1}\bigr).
\end{align*}
 Using $
s_{n+1}=e^{-r_n/2}s_n,$
we get that the tameness condition is equivalent to
\begin{align*}
&2^n(\alpha_n+2\beta_n-2\beta_{n+1})
-(k+2\ell)\log \sigma_n
+\ell\log \sigma_{n+1} \\
&\qquad
-(q+2m)\log s_{n+1}
+m\log s_{n+2} \\
&\qquad\qquad
\le
2^{n+1}\Bigl(
(1-e^{-r_n/2})s_n
-
(1-e^{-r_{n+1}/2})s_{n+1}
\Bigr).
\end{align*}

Since $s_n\to s_\infty>0$ and $r_n=n^{-1-\e}$, the right-hand side is equivalent to
\[
2^{n}s_\infty(r_n-r_{n+1})
\sim
\frac{2^n(1+\e)s_\infty}{2}{n^{-2-\e}}.
\]
As $\s_n=s_n-s_{n+1} \sim 2^{-1}s_\infty r_n$, the left-hand side is equivalent to
\[2^n(\alpha_n+2\beta_n-2\beta_{n+1}) =o(2^nn^{-2-\e})  \]
 Hence the estimate holds for large $n$.

To prove the second assertion, we estimate directly
\[
N_n(s_{n+1},s_n)
=
b_n\,\sigma_n^{-\ell}s_{n+1}^{-m}
\left(\frac{e^{s_{n+1}}}{e^{s_n}}\right)^{2^n}.
\]

Passing to the logarithm, we obtain
\begin{align*}
\log N_n(s_{n+1},s_n)
&=
2^n\beta_n
-\ell \log \sigma_n
-m\log s_{n+1}
+
2^n(s_{n+1}-s_n)\\
&\leq
-\ell \log \sigma_n
-m\log s_{n+1}
+
2^n(s_{n+1}-s_n)\\
&=
-\frac{2^{n-1}}{n^{1+\varepsilon}}s_\infty
+
O(\log n)
\end{align*}
 and therefore
\[
\log c_n-\log N_n(s_{n+1},s_n)=
\frac{2^{n-1}}{n^{1+\varepsilon}}s_\infty+o(2^nn^{-1-\e}) \to+\infty.
\]
Thus $
N_n(s_{n+1},s_n)=o(c_n).$\end{proof}
\subsection{KAM fixed point theorem}

   \begin{theorem}  Let $E$ be an Arnold chain, $(M,N)$ a KAM factor with phases $\a,\b$ satisfying $\a_n+2(\b_n-\b_{n+1})=o(1/n^{2+\e})$ for some $\e>0$ and let  $\rho \in \Bt^-$ be the Bruno sequence
with phase $1/n^{1+\e}$. Let $f_n:X_n \to X_{n+1}$ be a sequence of maps where $X_n \subset E_n$
   are arbitrary subsets. Assume that $f$ satisfies:
   $$|f_n(x)-f_{n-1}(y)|_s \leq M_n(s,t)|x-y|^2_t+N_n(s,t)|x-y|_t$$
   for any $s<t $.
    Then for any Bruno sequence $c \in \Bt^-$ whose phase is $o(n^{-1-\dt}),\ \dt>\e$ and any $t>0$, there exists $\eta>0$ such that the iterates 
$$x_{n+1}=f_n(x_n),\ s_n=t\hat \rho_{n-1}, x_0 \in \left( X_0 \cap E_t \right)$$ converge in $E_\infty$
provided that $|x_1-x_0|_{s_1} \leq \eta$.
  \end{theorem}
  \begin{proof}
    Set $u_n=|x_n-x_{n-1}|_{s_n}$. Then
\[
u_{n+1}
\le
M_n(s_{n+1},s_n)u_n^2
+
N_n(s_{n+1},s_n)u_n.
\]
By Proposition~\ref{P::KAM}, the pair
\[
(a_n,b_n):=(M_n(s_{n+1},s_n),N_n(s_{n+1},s_n))
\]
is tame. Hence, by Proposition~\ref{P::model}, there exists $\eta>0$
such that if $u_1\le \eta$, then $(u_n)$ converges to $0$ and is summable.
The sequence $(u_n)$ majorates $( | x_{n}-x_{n-1}|_{s_{n}})$, hence $(x_n)$ is a Cauchy sequence in $E_\infty$, therefore it converges.
  \end{proof}
 
  \medskip
  
\noindent\textbf{Acknowledgements.}
We are grateful to the referees for a careful reading of the manuscript and for helpful comments and suggestions.
   \bibliographystyle{plain}
\bibliography{master.bib}

\begin{thebibliography}{10}

\bibitem{Alazard_Shao}
T.~Alazard and C.~Shao.
\newblock {KAM via Standard Fixed Point Theorems}.
\newblock ArXiv:2312.13971, 2023.

\bibitem{Arnold_KAM}
V.I. Arnold.
\newblock { Proof of a theorem of A. N. Kolmogorov on the preservation of
  conditionally periodic motions under a small perturbation of the
  Hamiltonian}.
\newblock {\em Uspehi Mat. Nauk}, 18(5):13--40, 1963.
\newblock English translation: Russian Math. Surveys.

\bibitem{Arnold_edo}
V.I. Arnold.
\newblock {\em Chapitres suppl\'ementaires de la th\'eorie des \'equations
  diff\'erentielles ordinaires}.
\newblock MIR, 1980.

\bibitem{Bost_KAM}
J.-B. Bost.
\newblock {Tores invariants des syst\`emes dynamiques hamiltoniens}.
\newblock {\em Ast\'erisque}, 133-134:113-- 157, 1986.
\newblock S\'em. Bourbaki 639.

\bibitem{Bruno}
A.D. Bruno.
\newblock {Analytic form of differential equations I}.
\newblock {\em Trans. Moscow Math. Soc.}, 25:131--288, 1971.

\bibitem{Fejoz_KAM}
J.~F{\'e}joz.
\newblock {A proof of the invariant torus theorem of Kolmogorov}.
\newblock {\em Regular and Chaotic Dynamics}, 17(1):1--5, 2012.

\bibitem{Symplectic_torus}
M.~Garay, A.~Kessi, D.~van Straten, and N.~Yousfi.
\newblock {Quasi-periodic motions on symplectic tori}.
\newblock {\em Journal of Singularities}, 26:23--62, 2023.

\bibitem{Versal_fields}
M.~Garay and D.~van Straten.
\newblock {Versal deformations of isolated vector field singularities}.
\newblock arXiv:2011.06802, 2020.

\bibitem{Herman_conjecture}
M.~Garay and D.~van Straten.
\newblock {The Herman invariant tori conjecture}.
\newblock ArXiv:1909.06053v2, 2022.

\bibitem{Grothendieck_EVT}
A.~Grothendieck.
\newblock {\em Espaces vectoriels topologiques}.
\newblock Instituto de Matem\`atica Pura e Aplicada, Universidade de S\~ao
  Paulo, 1954.
\newblock 240 pp., English Translation: Topological vector spaces,\ Gordon and
  Breach, 1973.

\bibitem{Hamilton_implicit}
R.S. Hamilton.
\newblock {The inverse function theorem of Nash and Moser}.
\newblock {\em Bull. Amer. Math. Soc.}, 7(1):65--222, 1982.

\bibitem{Kolmogorov_KAM}
A.N. Kolmogorov.
\newblock {On the conservation of quasi-periodic motions for a small
  perturbation of the Hamiltonian function}.
\newblock {\em Dokl. Akad. Nauk SSSR}, 98:527--530, 1954.
\newblock (In Russian).

\bibitem{Nagumo}
M.~Nagumo.
\newblock { \"Uber das Anfangswertproblem partieller Differentialgleichungen}.
\newblock {\em Jap. J. Math.}, 18:41--47, 1942.

\bibitem{Sergeraert}
F.~Sergeraert.
\newblock {Un th\'eor\`eme de fonctions implicites sur certains espaces de
  Fr\'echet et quelques applications}.
\newblock {\em Ann. Sci. \'Ecole Norm. Sup.}, 5(4):599--660, 1972.

\bibitem{Siegel_vecteurs}
C.L. Siegel.
\newblock {\"Uber die Normalform analytischer Differentialgleichungen in der
  N\"ahe einer Gleichgewichtsl\"osung}.
\newblock {\em Nach. Akad. Wiss. G\"ottingen, math.-phys.}, pages 21--30, 1952.

\bibitem{Zehnder1974implicit}
E.~Zehnder.
\newblock An implicit function theorem for small divisor problems.
\newblock {\em Bull. Amer. Math. Soc.}, 80(4):174--179, 1974.

\bibitem{Zehnder_implicit}
E.~Zehnder.
\newblock {Generalized implicit function theorems with applications to some
  small divisor problems I}.
\newblock {\em Communications Pure Applied Mathematics}, 28:91--140, 1975.

\end{thebibliography}
\end{document}